\documentclass[onefignum,onetabnum]{siamart171218}
\definecolor{myGreen}{rgb}{0,0.5,0}
\newtheorem{prop}[theorem]{Proposition}
\usepackage{lipsum}
\usepackage{amsfonts}
\usepackage{graphicx}
\usepackage{epstopdf}
\usepackage{algorithmic}
\ifpdf
  \DeclareGraphicsExtensions{.eps,.pdf,.png,.jpg}
\else
  \DeclareGraphicsExtensions{.eps}
\fi

\newsiamremark{remark}{Remark}
\newsiamremark{hypothesis}{Hypothesis}
\crefname{hypothesis}{Hypothesis}{Hypotheses}
\newsiamthm{claim}{Claim}

\headers{Poisson-Nernst-Planck -- Finite domain effects}{D. Elad and N. Gavish}

\title{Finite domain effects in steady state solutions of Poisson-Nernst-Planck equations}

\author{Doron Elad\thanks{Department of Mathematics, Technion, Haifa 32000, Israel and Department of Applied Mathematics, Tel Aviv University, Tel Aviv 69978, Israel
  (\email{doronelad@mail.tau.ac.il}).}
\and Nir Gavish\thanks{Department of Mathematics, Technion, 923 Amado Bldg., Haifa 32000, Israel
  (\email{ngavish@technion.ac.il}).} }

\usepackage{amsopn}
\usepackage{soul}

\ifpdf
\hypersetup{
  pdftitle={Finite domain effects in steady-states of Poisson-Nernst-Planck equations},
  pdfauthor={D. Elad and N. Gavish}
}
\fi

\begin{document}

\maketitle

% REQUIRED
\begin{abstract}
Steady-state solutions of the Poisson-Nernst-Planck model are studied in the asymptotic limit of large, but finite domains.  By using asymptotic matching for integrals, we derive an approximate solution for the steady-state equation with exponentially small error with respect to the domain size.   The approximation is used to quantify the extent of finite domain effects over the full parameter space.
Surprisingly, already for small applied voltages (several thermal voltages), we found that finite domain effects are significant even for large domains (on the scale of hundreds of Debye lengths).   Namely, the solution near the boundary, i.e., the boundary layer (electric double layer) structure, is sensitive to the domain size even when the domain size is many times larger than the characteristic width of the boundary layer.    We focus on this  intermediate regime between confined domains and `essentially infinite' domains, and study how the domain size effects the solution properties. We conclude by providing an outlook to higher dimensions with applications to ion channels and porous electrodes.  
\end{abstract}

% REQUIRED

% REQUIRED

\section{Introduction}

Understanding the distribution of ions near charged surfaces is a fundamental problem in electrochemistry and biology, with a wide variety of applications, including water desalination, fuel cells, ion channels, and more.  Ions concentrate near charged interfaces, creating a layer of excess (counter-)charge that screens the surface charge.  This layer is known as the electrical double layer (EDL).  
The screening length, or the characteristic width of the electric double layer, is on the scale of a few nanometers in typical applications, and identifies with the so called Debye length for dilute electrolytes in large enough domains and near interfaces with low surface charge density.  

Originating from the seminal work of Nernst and Planck~\cite{nernst1889elektromotorische,planck1890ueber1,planck1890ueber2}, the Poisson-Nernst-Planck (PNP) model provides a theoretical basis for ion transport and redistribution in electrolytes.  
The PNP model and its generalizations had been extensively studied and used for electrolyte solutions, see e.g., \cite{bazant2005current,bazant2004diffuse,eisenberg1999structure,bolintineanu2009poisson,eisenberg1996computing,chen1993charges,eisenberg1998ionic,hollerbach2000predicting,hollerbach2001two,coalson2005poisson,dieckmann1999exploration,eisenberg2010crowded,chen1997permeation,bazant2004diffuse}, as well as other applications~\cite{markowich1990c}.  For a review of the different models, see~\cite{bazant2009towards} and references within. In many cases, studies focused on steady-state solutions of these models, e.g., current-voltage relations in ion channels \cite{naylor2016molecular} and electrochemical thin films \cite{bazant2005current}, salt adsorption in capacitive deionization cells \cite{porada2013review,biesheuvel2010membrane,suss2017size},  or capacitance of capacitors \cite{ji2014capacitance}.  

The PNP model with no-flux boundary conditions for the ionic species conserves the total concentration of each of the ionic species.  Therefore, the description of its steady-state involves the non-local constraint of total charge concentration.  Indeed, the steady-state of the PNP model is described by the Charge-Conserving-Poisson-Boltzmann (CCPB) equation~\cite{lee2010new,wan2014self,lee2015boundary,lee2014charge,beaume2011electrolyte} in which charge conservation is manifested by a non-local, integral, term.  
The non-local term couples the EDL region and the bulk: As ions concentrate near the boundary to screen charge, they are depleted from the interior of the domain.  For large enough domains, the amount of ions depleted from the interior is negligible with respect to the overall amount of ions, and therefore finite domain effects are negligible.  In this case,  the steady-state of the PNP model can be approximated by the solution of Poisson-Boltzmann (PB) equation, which involves only local terms.  On the other extreme, in confined domains, on the magnitude of several Debye lengths, the EDLs from the two boundaries overlap and so finite domain effects are dominant.  Here, we show that there exists an intermediate regime of large domains in which finite domain effects are significant.

In this work, we study steady-state solutions of the Poisson-Nernst-Planck equations in finite domains.   
Particularly, we focus on the asymptotic regime of a large, but finite, domain size.   
Previous works have considered this asymptotic regime by boundary layer analysis~\cite{chen1997qualitative,wang2014singular}.  Here, we take a different approach, and directly approximate the singular integrals appearing in the non-local CCPB equation.  The result is a complementary analysis that gives rise to an accurate approximation of the solution with exponentially small error with respect to the domain size parameter, and that highlights the role of finite-domain effects.  Surprisingly, we find that even for relatively large domains, on a magnitude of hundreds of Debye lengths, finite domain effects are significant.  These results are relevant for numerical simulations in which PNP models in infinite domains are approximated by PNP models in large, but finite domains, and for a wide range of applications, including ion channels~\cite{eisenberg2010crowded}, dendritic spines~\cite{matsuzaki2004structural,yuste1995dendritic}, submicron gap capacitors~\cite{hourdakis2006submicron,wallash2003electrical,chen2006electrical}, and microfluidics~\cite{whitesides2006origins}.

The paper is organized as follows. In \Cref{sec:model}, we provide a brief mathematical review of the PNP model and of the CCPB equation that describes its steady-state.  In~\Cref{sec:perturbation} we use asymptotic matching methods to obtain a highly accurate approximation of the steady-state solution of PNP in a finite domain with an error that decreases exponentially with domain size.  In~\Cref{sec:finiteDomainEffects} this approximation is used to reveal when finite domain effects are significant and to quantify their nature.  Particularly, in~\Cref{sec:distinct}, we identify three parameter regimes: A region corresponding to confined domains, a region corresponding to very large domains in which finite-domain effects are negligible, and an intermediate regime where the the domain is large enough so that the solution reaches an electroneutral bulk, but finite domain effects are yet significant.  In~\Cref{sec:screening}, we further quantify the effects of finite domain size on the screening length.  Our results show that even for relatively large domains, on the magnitude of hundreds of Debye lengths, finite domain effects are significant.  
A methodology to numerically study finite domain effects in generalized PNP models is presented in~\Cref{sec:diffBC}, and demonstrated on the PNP-Stern model.  In \Cref{sec:highdim} we provide an outlook to higher dimensions, and demonstrate applications for ionic channels and porous electrodes.
Concluding remarks are presented in \Cref{sec:conc}.

\section{Model}\label{sec:model}
We consider the Poisson-Nernst-Planck (PNP) model for a 1:1 ionic solution bounded between two electrodes located at~$x=\pm L/2$,
\begin{equation} \label{eq:PNPdimensional}
\begin{split}
&p_t(t,x)=-  \frac{\partial}{\partial x} J_p, \qquad J_p = -D\left[p_x +p \frac{q}{k_B T}  \phi_x\right],\\
&n_t(t,x)=-  \frac{\partial}{\partial x} J_n , \qquad J_n = -D\left[n_x-n \frac{q}{k_B T}  \phi_x\right],\\
&-\epsilon_0\epsilon_r \phi_{xx}(x) = q \left( p-n \right) , \qquad  -\frac{L}{2}<  x < \frac{L}{2},\quad t>0,
\end{split}
\end{equation}
with the initial conditions 
	\begin{equation}
	p(x,0)=p_0(x), \quad n(x,0)=n_0(x), \quad \phi(x,0)=\phi_0(x).
	\end{equation}
Here~$\phi$ is the electrostatic potential, $p$ and~$n$ are the concentrations of positively and negatively charged ions, respectively.  Additionally, $D$ is the diffusion coefficient (assumed to be equal for the two ionic species), $q$ is the elementary charge, $k_B$ is the Boltzmann constant, $T$ is temperature,~$\epsilon_0$ is the vacuum permittivity and~$\epsilon_r$ is the relative permittivity.  

In what follows, we introduce the non-dimensional variables
	\begin{equation}\label{eq:scaledVars}
		\tilde{x}=\frac{x}{\lambda_D} , \quad \tilde{\phi}=\frac{q \phi}{k_B T}, \quad \tilde p=\frac{p}{\bar{c}}, \quad \tilde n=\frac{n}{\bar{c}},\quad \quad \tilde{L}=\frac{L}{\lambda_D}, \quad \lambda_D=\sqrt{\frac{\epsilon_0\epsilon_r k_B T }{2 \bar{c}q^2 }},
	\end{equation}
	where~$\lambda_D$ is the Debye length and~$\bar{c}$ is the average initial ionic concentrations, which is assumed to be equal for the two ionic species, i.e., the initial ionic concentrations are chosen so that the electrolyte solution is globally electroneutral
\begin{equation}\label{eq:barc}
\bar{c}:=\frac{1}{L} \int_{-L/2}^{L/2} p_0(x)dx=  \frac{1}{L} \int_{-L/2}^{L/2} n_0(x)dx.
\end{equation}
The non-dimensional version of PNP~\eqref{eq:PNPdimensional} (presented after omitting the tildes) reads as
\begin{subequations}\label{eq:PNP}
\begin{equation}\label{eq:NernstPlanck}
p_t=-  \frac{\partial}{\partial x} J_p,\quad J_p=p_x+p\phi_x,\qquad n_t=-  \frac{\partial}{\partial x} J_n,\quad J_n=n_x-n\phi_x, 
\end{equation}
\begin{equation}\label{eq:Poisson}
\phi_{xx}=\frac{n-p}2,
\end{equation}
in the domain~$-\frac{L}{2} < x < \frac{L}{2}$ with initial conditions
\begin{equation}	
p(x,0)=p_0(x), \quad n(x,0)=n_0(x), \quad \phi(x,0)=\phi_0(x),
\end{equation}
that satisfy 
\begin{equation}\label{eq:barc_normalized}
\frac{1}{L} \int_{-L/2}^{L/2} p_0(x)dx=  \frac{1}{L} \int_{-L/2}^{L/2} n_0(x)dx=1.
\end{equation}

In this work, we focus on no-flux boundary conditions, and fixed applied voltage~$V$ on the electrodes
	\begin{equation}\label{eq:PNP_noflux_phiBC}
J_p(x=\pm L/2,t)=J_n(x=\pm L/2,t)=0,\quad \phi(\pm L/2,t)=\pm V.
	\end{equation}
	\end{subequations}
See Section~\ref{sec:diffBC} for additional cases.  

Equations~\eqref{eq:PNP} have a unique steady-state solution which satisfies the charge conserving Poisson-Boltzmann (CCPB) equation~\cite{lee2014charge,lee2010new,wan2014self}
	\begin{equation} \label{eq:CCPB}
	\phi_{xx}= \frac{ \sinh \phi }{\frac{1}{L} \int_{-\frac{L}{2}}^{\frac{L}{2}} e^{\phi}dx} , \qquad \phi(\pm L/2)=\pm V.
	\end{equation}
For completeness, we provide here a brief review of the derivation of the CCPB equation, and refer the reader to~\cite{lee2010new,wan2014self} for more details:  Under no-flux boundary conditions, see~\eqref{eq:PNP_noflux_phiBC}, 
the PNP equations~\eqref{eq:PNP} preserve the average ionic concentration of each ion during the system evolution for all $t>0$, and therefore, in accordance with~\eqref{eq:barc_normalized}, the steady-states solutions satisfy
\begin{equation}\label{eq:avgconc_time}
\frac{1}{L} \int_{-L/2}^{L/2} p (x)dx=  \frac{1}{L} \int_{-L/2}^{L/2} n (x)dx=1. 
\end{equation}
 Furthermore, at steady-state, the Nernst-Planck equations~\eqref{eq:NernstPlanck}  reduce to
$$
J_p = p_x+p\phi_x=0,\quad J_n = n_x-n\phi_x=0,
$$
and can be integrated to yield 

	\begin{equation}\label{eq:Boltzmann1}
	p(x)=\alpha e^{-\phi(x)},\quad n(x)=\beta e^{\phi(x)},
	\end{equation}
where $\alpha$ and $\beta$ are integration constants.  Taking the average of both sides in each of the equations in~\eqref{eq:Boltzmann1}, substituting~\eqref{eq:avgconc_time}, and isolating~$\alpha$ and $\beta$ yields
\begin{equation}\label{eq:alphaBeta}
\alpha=\frac{1}{\frac1L\int_{-L/2}^{L/2} e^{-\phi(x)}dx},\quad \beta=\frac{1}{\frac1L\int_{-L/2}^{L/2} e^{\phi(x)}dx}.
\end{equation}
The following Lemma follows from~\cite[Theorem 1.2]{lee2010new},  and is used frequently in this work.
\begin{lemma}\label{lemma:LemmaOnCCPB} 
	The solution $\phi$ of equation \eqref{eq:CCPB} is odd and monotonically increasing. In particular,
	$\phi(0)=0$ 
and
\begin{equation}\label{eq:alphaBeta_equal} 
\alpha=\frac{1}{\frac1L\int_{-L/2}^{L/2} e^{-\phi(x)}dx}=\frac{1}{\frac1L\int_{-L/2}^{L/2} e^{\phi(x)}dx}=\beta.
\end{equation}
\end{lemma}
Symmetry considerations, in particular, imply that~$\alpha=\beta$, see~\eqref{eq:alphaBeta_equal}.  Finally, substituting~\eqref{eq:Boltzmann1} and~\eqref{eq:alphaBeta_equal} into Poisson's equation~\eqref{eq:Poisson} yields the CCPB equation~\eqref{eq:CCPB}.

In the limit of an infinite domain size,~$L\to\infty$, the charge-conserving Poisson-Boltzmann equation~\eqref{eq:CCPB} reduces to the classical Poisson-Boltzmann equation,~see~\cite{lee2010new} and also Section~\ref{sec:formulation},
\begin{equation}\label{eq:PB}
\phi_{xx}=\sinh\phi.
\end{equation}
In this case,~$\alpha=\beta=1$.
The point~$x=0$ is farthest from the boundaries, i.e., it is in the bulk of the electrolyte solution.    
For sufficiently large domain size,~$\phi(x)\approx0$ for~$x=O(1)$ and hence the parameter,~$\alpha$, can be identified as the (normalized) ion concentration in the bulk, see~\eqref{eq:Boltzmann1}.

\subsection{Equivalent formulation of the CCPB equation}\label{sec:formulation}

We consider an equivalent formulation of the CCPB equation~\eqref{eq:CCPB} in terms of its inverse function~$x(\phi)$.  A similar formulation is available for the PB equation~\eqref{eq:PB}, see, e.g., \cite[Appendix A.1]{schmickler2010interfacial}.  Here we briefly review the derivation details adapted to the CCPB equation~\eqref{eq:CCPB}.  

Multiplying~\eqref{eq:CCPB} by $\frac{\partial \phi}{\partial x}$, integrating and using the monotonicity of $\phi$ (Lemma \ref{lemma:LemmaOnCCPB}) implies
\begin{equation}
\frac{\partial \phi}{\partial x}=\sqrt{  \frac{ 2 }{\frac{1}{L} \int_{-\frac{L}{2}}^{\frac{L}{2}} e^{\phi}dx}   \cosh \phi + C  },
\end{equation}
where $C$ is a constant of integration.
Inserting  $\phi(0)=0$ (see Lemma \ref{lemma:LemmaOnCCPB}) implies
\begin{equation}
C=\phi_x^2(0)-\frac{ 2 }{\frac{1}{L} \int_{-\frac{L}{2}}^{\frac{L}{2}} e^{\phi}dx},
\end{equation}
which yields the following formula for the inverse of $\phi(x)$ for~$0<x<L/2$,
	\begin{equation} \label{eq:implicit}
	x\left(\phi \right)=
	\int_0 ^\phi \frac{d\xi}{\sqrt{4 \alpha \sinh^2 \left(\frac{\xi}{2}\right)+\phi_x^2 (0)}} , \qquad x(V)=\frac{L}{2},
	\end{equation}
	where~$\alpha$~is given by~\eqref{eq:alphaBeta}, or equivalently
	\begin{equation} \label{eq:xofPhi2}
	x\left(\phi \right)=\frac{L}{2}
	\frac{\int_0 ^\phi \frac{1}{\sqrt{ \sinh^2 \left(\frac{\xi}{2}\right)+\frac{\phi_x^2 (0)}{4\alpha}}} d\xi }{\int_0 ^V \frac{1}{\sqrt{ \sinh^2 \left(\frac{\xi}{2}\right)+\frac{\phi_x^2 (0)}{4\alpha}}} d\xi},\quad 0\le \phi \le V.
	\end{equation}
Since~$\phi(x)$ is an odd function, see Lemma~\ref{lemma:LemmaOnCCPB}, the values of~$x(\phi)$ for~$-V\le \phi\le 0$ are readily defined via the two equivalent relations~\eqref{eq:implicit} and~\eqref{eq:xofPhi2}.
The latter formulation reveals that, given the values of $L, \, V$, the solution of \eqref{eq:implicit} depends on the ratio $\phi_x^2 (0)/
\alpha$, rather than on the two separate quantities $\phi_x(0)$ and~$\alpha$.  
Accordingly, let us define
\begin{equation}\label{eq:eps_def}
\varepsilon:=\frac{\phi_x(0)}{2\sqrt{\alpha}}.
\end{equation}
The observation that the solution of \eqref{eq:implicit} depends on~$\varepsilon$, rather than on two independent quantities enables and motivates the study of the solution in the asymptotic regime of small~$\varepsilon$.
In what follows, we will show that~$0<\varepsilon\ll1$ corresponds to a large domain size, and focus on this regime.
\section{Approximation of $x(\phi)$ by singular perturbation theory} \label{sec:perturbation}
Equation~\eqref{eq:implicit} defines the inverse steady-state solution~$x(\phi)$ of~\eqref{eq:PNP} in terms of $\alpha$ and~$\phi_x(0)$, rather than solely as a function of the natural problem parameters, the domain size~$L$ and the applied voltage~$V$.    We now use asymptotic analysis to approximate the inverse steady-state solution given the parameters~$L$ and~$V$.
\subsection{Problem formulation}
In the case of an infinite domain, $L=\infty$, one obtains that $\phi_x (0)=0$, where~$\phi(x)$ is the steady-state solution of~\eqref{eq:PNP}, see \cite[Section 3.1]{gavish2016structure}.  In this case,~$\varepsilon=0$, see~\eqref{eq:eps_def}, hence~\eqref{eq:xofPhi2} implies that~$\phi(x)$ satisfies the Poisson-Boltzmann equation~\eqref{eq:PB} with~$\alpha=1$.  Accordingly, a large, but finite, domain size~$L\gg1$, corresponds to a regime when~$0<\varepsilon\ll1$, $|\phi_x(0)|\ll 1$ and $\left| \alpha-1\right|\ll 1$.  
In what follows, we focus on this regime of a large domain size.  For convenience, we rewrite~\eqref{eq:implicit} in terms of an  integral~$I(\phi;\varepsilon)$ with singular behavior as~$\varepsilon\to0$,
\begin{subequations}\label{eq:intRepresentation2}
\begin{equation}
	x(\phi) = \frac{1}{2\sqrt{\alpha}} I\left( \phi; \varepsilon\right), \qquad 	 x(V)=\frac{L}{2}, \qquad \alpha=\left( \frac{1}{L} \int_{-L/2}^{L/2} e^\phi dx \right)^{-1},
	\end{equation}
where 
\begin{equation} \label{eq:defOfI}
I(\phi;\varepsilon):= \int_0 ^\phi \frac{dx}{\sqrt{ \sinh^2 \left(x/2 \right) + \varepsilon^2}}.
\end{equation}
 \end{subequations}

\subsection{Evaluation of the integral~$I(\phi;\varepsilon)$ with singular behavior}
The integral~$I(\phi;\varepsilon)$, see~\eqref{eq:defOfI}, diverges as~$\varepsilon\to0$.  This is an inherent property of the problem formulation, since~$\varepsilon=0$ corresponds to an infinite domain size, 
while~$\varepsilon>0$ corresponds to a finite domain.

The following proposition evaluates the integral~$I(\phi;\varepsilon)$ for~$0<\varepsilon\ll1$ using range splitting and asymptotic matching, while exploiting the fact that the integrand of $I(\phi;\varepsilon)$ behaves differently in an inner region near $x=0$ where~$\sinh x \ll \varepsilon$ and in an outer region where~$\sinh x \gg \varepsilon$.
\begin{prop} \label{prop:expansion}
	Let $0<\varepsilon\ll 1$. Then, the integral~\eqref{eq:defOfI} satisfies
	\[
	I(\phi;\varepsilon)=I^{\rm approx}(\phi;\varepsilon)+E(\phi),
	\]
	where
	\begin{equation}\label{eq:LemmaForI}
	I^{\rm approx}(\phi;\varepsilon)=\begin{cases}
	2\,{\rm arcsinh}\left(\frac{\phi}{2\varepsilon}\right), & 0\le \phi \leq \eta(\varepsilon)\\
	2\log\frac{4}{\varepsilon}-4\,{\rm arccoth}\left(e^{\frac{\phi}2}\right), & \phi > \eta(\varepsilon)
	\end{cases} ,
	\end{equation}
$\eta(\varepsilon)=\varepsilon^{\frac34}$ and the error~$|E(\phi)|\le c\,\sqrt\varepsilon$ where~$c$ is a constant.
\end{prop}
\begin{proof}
For all~$0<\eta<\phi$,~$I(\phi;\varepsilon)=I_1(\eta;\varepsilon)+I_2(\phi;\varepsilon)$, where
\begin{equation}\label{eq:I0I1}
I_1(\eta;\varepsilon)=\int_0 ^\eta \frac{dx}{\sqrt{ \sinh^2 \left( x/2 \right) + \varepsilon^2}},\quad I_2(\phi;\varepsilon)=\int_\eta^\phi \frac{dx}{\sqrt{ \sinh^2 \left( x/2 \right) + \varepsilon^2}}.
\end{equation}
To approximate~$I_1$, we focus on the regime~$\eta\ll1$ where
\[
\sinh^2 \left( \frac{x}2 \right)=\frac{x^2}4+O(x^4),\quad 0\le x\le \eta.
\]
In what follows, it becomes clear that choosing any~$\eta\ll1$, e.g.,~$\eta=
\varepsilon$ or~$\eta=\sqrt\varepsilon$, is insufficient to control the approximation error, and rather a more careful consideration of the choice of~$\eta$ is required.  Indeed, since 
\begin{equation}\label{eq:approxI0_sinh}
\sinh^2 \left( \frac{x}2 \right)+\varepsilon^2=\frac{x^2}4+\varepsilon^2+\tilde E_1(x),\quad \tilde  E_1(x)=O(x^4),
\end{equation}
we further restrict~$x^4\ll \varepsilon^2$ so that the error satisfies both~$ \tilde E_1(x)\ll \varepsilon^2$ and~$\tilde E_1(x)\ll x^2 $.
Let $0<\eta\ll \sqrt{\varepsilon}$. Substituting~\eqref{eq:approxI0_sinh} in~\eqref{eq:I0I1} yields
\begin{equation*}\begin{split}
I_1(\phi;\varepsilon)&=\int_0 ^\eta \frac{dx}{\sqrt{ \frac{x^2}4  + \varepsilon^2+\tilde E_1(x)}}=
\int_0 ^\eta \frac{dx}{\sqrt{ \frac{x^2}4  + \varepsilon^2}\sqrt{ 1+\frac{\tilde E_1(x)}{\frac{x^2}4  + \varepsilon^2}}}\\&=\frac{1}{\sqrt{ 1+\frac{\tilde E_1(\xi)}{\frac{\xi^2}4  + \varepsilon^2}}}\int_0 ^\eta \frac{dx}{\sqrt{ \frac{x^2}4  + \varepsilon^2}},\quad 0<\xi<\eta,
\end{split}\end{equation*}
where the last equality follows from the Mean Value Theorem for integrals.
Thus,
\begin{subequations}\label{eq:I1E1}
	\begin{equation} \label{eq:I1}
	I_1=I_1^{\rm approx}(\phi;\varepsilon)+E_1,\quad I_1^{\rm approx}(\phi;\varepsilon)= \int_0 ^\phi \frac{dx}{\sqrt{\varepsilon^2 + \frac{x^2}{4}}}=2\,\mbox{arcsinh}\left(\frac{\phi}{2\varepsilon}\right) \end{equation}
	where the error takes the form
	\begin{equation}\label{eq:E1}
	E_1=2\,\mbox{arcsinh}\left(\frac{\phi}{2\varepsilon}\right)\left\{1-\left[ 1+\frac{\tilde E_1(\xi)}{\frac{\xi^2}4  + \varepsilon^2}\right]^{-\frac12}\right\},\quad 0<\xi<\eta.
	\end{equation}
	\end{subequations}
The approximation of~$I_2$ leads to the range~$\varepsilon\ll\eta\ll \sqrt{\varepsilon}$, see below.  Let us consider~$\eta=c_\eta\,\varepsilon^{\frac34}$ where~$c_\eta$ is a constant.   Denote the bound~$M_1$ such that~$|\tilde E_1|\le M_1\eta^4$, see~\eqref{eq:approxI0_sinh}.  Then, 
for~$0<\phi<\eta=c_\eta\,\varepsilon^{\frac34}$, the error~\eqref{eq:E1} is bounded by
\begin{equation}\label{eq:E1bound}
|E_{1}|\le 2\,\mbox{arcsinh}\left(\frac{\eta}{2\varepsilon}\right)\left|\frac{\tilde E_1(\xi)}{\frac{\xi^2}4  + \varepsilon^2}\right|\le \frac{2c_\eta\eta}{2\varepsilon}\frac{M_1\eta^4}{\varepsilon^2}=c_\eta M_1\varepsilon^{\frac34}.
\end{equation}

To approximate~$I_2$, see~\eqref{eq:I0I1}, we focus on the regime~$\eta\gg\varepsilon$ where~$\sinh^2 \left(x/2\right)\gg\varepsilon^2$.
In this case,
\[
I_2(\phi;\varepsilon)=\left[1+\frac{\varepsilon^2}{\sinh^2\left( \frac{\xi}2 \right)}\right]^{-\frac12}\int_{\eta}^{\phi}\frac{dx}{\sinh\left( \frac{x}2 \right)},\quad \eta<\xi<\phi.
\]
Thus,
\begin{subequations}\label{eq:I2E2}
	\begin{equation} \label{eq:I2}\begin{split}
&I_2(\phi;\varepsilon)=I_2^{\rm approx}(\phi;\varepsilon)+E_2,\quad \\&I_2^{\rm approx}(\phi;\varepsilon)=\int_{\eta}^{\phi}\frac{dx}{\sinh\left( \frac{x}2 \right)}=
4\left[\mbox{arccoth}\left(e^{\frac\eta2}\right)-\mbox{arccoth}\left(e^{\frac\phi2}\right)\right],
\end{split}
	\end{equation} 
where
	\begin{equation} \label{eq:E2}
E_2=4\left[\mbox{arccoth}\left(e^{\frac\eta2}\right)-\mbox{arccoth}\left(e^{\frac\phi2}\right)\right]\left\{\left[1+\frac{\varepsilon^2}{\sinh^2\left( \frac{\xi}2 \right)}\right]^{-\frac12}-1\right\}.
	\end{equation} 
	\end{subequations}
Therefore, for~$\phi>\eta$
\begin{equation}\label{eq:I_phi_bigger_than_eta}
I^{\rm approx}(\phi)\approx I_1^{\rm approx}(\eta)+I_2^{\rm approx}(\phi)=2\log\frac{4}{\varepsilon}-4{\rm arccoth}\left(e^{\frac{\phi}2}\right)+\frac{2\sqrt{\varepsilon}}{c_\eta^2}+O(\varepsilon),\quad 0<\varepsilon\ll1.
\end{equation}
The errors~\eqref{eq:E1} and~\eqref{eq:E2} contribute to the the approximation error in~\eqref{eq:I_phi_bigger_than_eta}.
Expansion of~\eqref{eq:E1} and~\eqref{eq:E2} for~$0<\varepsilon\ll1$, shows that 
\[
|E_1(\eta)+E_2(\phi)|=O(\varepsilon|\log\varepsilon|).
\]
Additionally, since~$I(\phi)$ is independent on~$\eta$, the third term in~\eqref{eq:I_phi_bigger_than_eta} also contributes to the error
\[
E(\phi)=O\left(E_1(\eta)+E_2(\phi),\sqrt{\varepsilon}\right)=O\left(\sqrt{\varepsilon}\right).
\]
Finally, expression~\eqref{eq:LemmaForI} is attained by (arbitrarily) setting~$c_\eta=1$ and resolving
\[
I^{\rm approx}(\phi)=\begin{cases}
I_1^{\rm approx}(\phi)& 0\le \phi\le\varepsilon^{\frac34},\\
2\log\frac{4}{\varepsilon}-4{\rm arccoth}\left(e^{\frac{\phi}2}\right)&\varepsilon^{\frac34}<\phi,
\end{cases}
\]
where~$I_1^{\rm approx}(\phi)$ is given by~\eqref{eq:I1}
\end{proof}

\subsection{Numerical validation and a refined approximation}	
Proposition~\ref{prop:expansion} provides an approximation to~\eqref{eq:defOfI} with~$O\left(\sqrt{\varepsilon}\right)$ error.  \begin{figure}[h!]
\includegraphics[width=\textwidth]{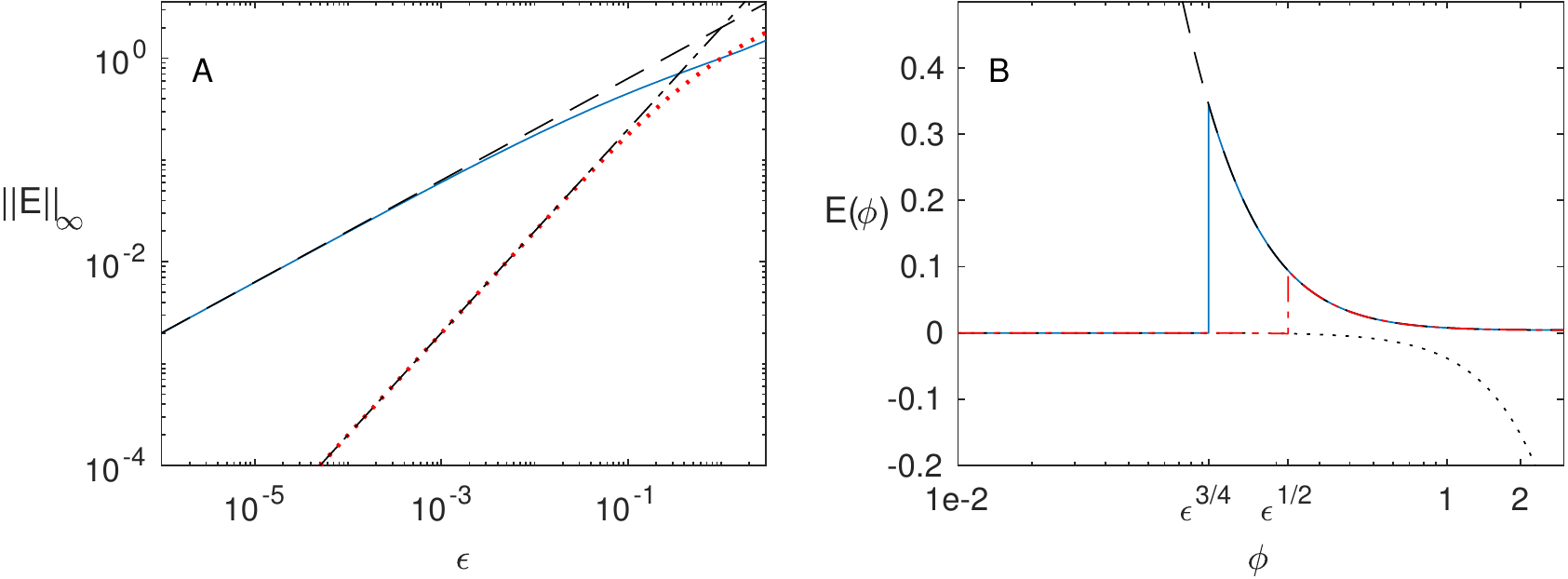}
\caption{A: Error~$\|E\|_\infty$ of approximation~\eqref{eq:LemmaForI} with~$\eta=\varepsilon^{3/4}$ as a function of~$\varepsilon$ ({\color{blue} solid}), the curve~$2\sqrt\varepsilon$ (dashes), and error~$\|E\|_\infty$ of approximation~\eqref{eq:LemmaForI} with~$\eta=\varepsilon^{1/2}$ as a function of~$\varepsilon$ ({\color{red} dots}) and the curve~$2\varepsilon$ (dash-dots).  B:  Error~$E(\phi;\varepsilon=0.05)$ of approximation~\eqref{eq:LemmaForI}  with~$\eta=\varepsilon^{3/4}$ ({\color{blue} solid}) and $\eta=\varepsilon^{1/2}$  ({\color{red} dash-dots}).  Additional curves are the error components~$I(\phi)-I_1(\phi)$ (dots) and~$I(\phi)-\left[2\log\frac{4}{\varepsilon}-4{\rm arccoth}\left(e^{\frac{\phi}2}\right)\right]$ (dashes), see~\eqref{eq:I1} and~\eqref{eq:I2}.}
\label{fig:crudeError}
\end{figure}
Figure~\ref{fig:crudeError}A shows that, as expected, the error magnitude is~$O\left(\sqrt{\varepsilon}\right)$.  Figure~\ref{fig:crudeError}B presents a profile of the error~$E(\phi)$ for~$\varepsilon=0.05$.  Consistent with the proof of Proposition~\ref{prop:expansion}, we observe that the error is maximal at the matching region.   Furthermore, Figure~\ref{fig:crudeError}B strongly suggests that the choice of the matching region~$\eta$ is not optimal, and rather one should consider a larger~$\eta$.  Indeed, 
the following proposition refines the choice of the matching region~$\eta$ in Proposition~\ref{prop:expansion} by considering higher-order correction terms.  
\begin{prop} \label{prop:expansion_refined}
	Let~$0<\varepsilon\ll 1$.  
	The error~$E(\phi)$ in the approximation~\eqref{eq:LemmaForI} with $\eta(\varepsilon)=\sqrt{\varepsilon}$ satisfies
	\[
	|E(\phi)|=\begin{cases}
	\frac{\varepsilon^2}2 \frac{\cosh\frac{\phi}2}{\sinh^2\frac{\phi}2}+O(\varepsilon^2\log\varepsilon), & \phi>\eta(\varepsilon) , \\
	-\frac{\phi^2}{24} + O(\varepsilon^2), & \phi\leq \eta(\varepsilon).
	\end{cases}
	\]
	In particular,~$|E(\phi)|=O(\varepsilon)$.
\end{prop}
\begin{proof}
The proof extends the analysis involved in the proof of Proposition~\ref{prop:expansion} by considering the higher order terms of the relevant expansions, and by refining the matching region.  See Appendix~\ref{app:lem_expansion_refined} for details.
\end{proof}
Figure~\ref{fig:crudeError} shows that, as expected, the choice of~$\eta=\sqrt\varepsilon$, rather than~$\eta=\varepsilon^{3/4}$, leads to a much smaller error, with magnitude~$O(\varepsilon)$.
In retrospect, the analysis leading to Proposition~\ref{prop:expansion} led to a non-optimal choice of~$\eta$ due to a restriction~$\eta\gg\sqrt\varepsilon$ required to ensure the approximation~\eqref{eq:approxI0_sinh} is accurate.  Namely, to ensure that the~$O(x^4)$ term in~\eqref{eq:approxI0_sinh} is significantly smaller than each of the other terms.  Considering the explicit error term,~see~\eqref{eq:explicitOx4}, rather than just its~$O(x^4)$ magnitude, enables relaxing the restriction~$\eta\gg\sqrt\varepsilon$.  Removing this restriction allowed improving  approximation~\eqref{eq:LemmaForI} merely by modifying~$\eta$, rather than also accounting for high-order corrections.
\subsection{Sole dependence on problem parameters $L$ and $V$}\label{sec:aAndE}
Proposition~\ref{prop:expansion_refined} can be applied to approximate the inverse steady-state solution $x(\phi)$ of the PNP system~\eqref{eq:PNP}.  To do so, let us first consider the relation between~$\varepsilon$ and~$\alpha$ to the natural problem parameters $L$ and $V$.  
In what follows, we restrict our attention to the case $V=O(1)$ as~$\varepsilon\to0$. 
\begin{prop} \label{prop:relationAlpha}
Let~$V\ne0$, and ~$0<\varepsilon\ll1$. In addition,  let~$x(\phi)$ and~$\alpha$ be defined by \eqref{eq:intRepresentation2}. Then,
\begin{equation} \label{eq:approxForSqrta}
\sqrt{\alpha} = \sqrt{1+\left[\frac{4\sinh^2(V/4)}{L}\right]^2}-\frac{4\sinh^2(V/4)}{L}+ O\left( \varepsilon\right),
\end{equation}
and
\begin{equation} \label{eq:firstRelation}
\varepsilon=4\tanh\left(\frac{V}4\right)e^{-L\sqrt{\alpha}/2 } \left( 1+O(\varepsilon) \right).
\end{equation}	
\end{prop}
\begin{proof}
The definition of $\alpha$, see \eqref{eq:intRepresentation2}, implies that
\[
\alpha=\left( \frac{1}{L} \int_{-V}^{V} e^\phi \frac{dx}{d\phi} d\phi \right)^{-1}=L\left( \frac1{2\sqrt\alpha}\int_0^V \frac{e^{\phi}+e^{-\phi}}{\sqrt{\sinh^2 \left( \frac{\phi}{2}\right)+\varepsilon^2}} d\phi\right)^{-1},
\]
hence
\begin{equation} \label{eq:secondRelation}
\sqrt{\alpha} = L\left( \int_0^V \frac{\cosh(\phi)}{\sqrt{\sinh^2 \left( \frac{\phi}{2}\right)+\varepsilon^2}} d\phi\right)^{-1}.
\end{equation}
The integral in~\eqref{eq:secondRelation} is approximated using range splitting and asymptotic matching, similar to the proof of Proposition \ref{prop:expansion_refined}.
Particularly, 
\begin{equation}\begin{split}
\int_0^V \frac{\cosh(\phi)d\phi}{\sqrt{\sinh^2 \left( \frac{\phi}{2}\right)+\varepsilon^2}} &=\underbrace{\int_0^{\sqrt\varepsilon} \frac{\cosh(\phi)d\phi}{\sqrt{\sinh^2 \left( \frac{\phi}{2}\right)+\varepsilon^2}}}_{I_1^*} +\underbrace{\int_{\sqrt\varepsilon}^V \frac{\cosh(\phi)d\phi}{\sqrt{\sinh^2 \left( \frac{\phi}{2}\right)+\varepsilon^2}}}_{I_2^*}.
\end{split}
\end{equation}
The integral~$I_1^*$ can be approximated by the integral~$I_1$~\eqref{eq:I0I1}.  Indeed, the Mean Value Theorem for integrals implies that
\[
\int_0^{\sqrt\varepsilon} \frac{\cosh(\phi)d\phi}{\sqrt{\sinh^2 \left( \frac{\phi}{2}\right)+\varepsilon^2}}=\cosh(\xi)\int_0^{\sqrt\varepsilon} \frac{d\phi}{\sqrt{\sinh^2 \left( \frac{\phi}{2}\right)+\varepsilon^2}},\quad 0<\xi<\sqrt{\varepsilon}.
\]
Since,~$\cosh(\xi)<1+\varepsilon$ for~$0<\xi<\sqrt{\varepsilon}$ and for small enough~$\varepsilon$,
\[
I_1^*=(1+O(\varepsilon))I_1(\sqrt\varepsilon;\varepsilon).
\]
Similarly, using the same argument as in the approximation of~\eqref{eq:I0I1} in Proposition \ref{prop:expansion_refined},
\begin{equation*}\begin{split}
I_2^*&=I_2(V;\varepsilon)-\int_{\sqrt\varepsilon}^V \frac{1-\cosh(\phi)}{\sinh(\phi/2)}d\phi+O(\varepsilon)\\&=I_2(V;\varepsilon)+4\left[\cosh(V/2)-\cosh(\sqrt\varepsilon/2)\right]+O(\varepsilon).
\end{split}
\end{equation*}
Therefore,
\begin{equation}\label{eq:Istar}\begin{split}
\int_0^V \frac{\cosh(\phi)d\phi}{\sqrt{\sinh^2 \left( \frac{\phi}{2}\right)+\varepsilon^2}}&=I^{\rm approx}(V;\varepsilon)+4\left[\cosh(V/2)-\cosh(\sqrt\varepsilon/2)\right]+O(\varepsilon)\\&=I^{\rm approx}(V;\varepsilon)+8\sinh^2(V/4)+O(\varepsilon),
\end{split}
\end{equation}
Substituting~\eqref{eq:Istar} into~\eqref{eq:secondRelation} gives rise to~\eqref{eq:approxForSqrta}. 

By~\eqref{eq:intRepresentation2},
\begin{equation}\label{eq:raw_eps}
2\sqrt{\alpha}x(V)=\sqrt{\alpha}L=I(V;\varepsilon)=2\log\frac{4}{\varepsilon}-4\,{\rm arccoth}\left(e^{\frac{V}2}\right)+O(\varepsilon).
\end{equation}
Isolating~$\varepsilon$ in~\eqref{eq:raw_eps} gives rise to~\eqref{eq:firstRelation}. 
\end{proof}
Propositions~\ref{prop:expansion_refined} and~\ref{prop:relationAlpha} can be readily applied to approximate the inverse steady-state solution $x(\phi)$ of PNP~\eqref{eq:PNP},
\begin{lemma} \label{lem:explicitXOfPhi}
	Let $L\gg 1,~V\ne0$, and let~$\phi(x;L,V)$ be a steady-state solution of the PNP system~\eqref{eq:PNP}. Then, the inverse function~$x(\phi)=\phi^{-1}(x)$ satisfies 
		\begin{subequations}\label{eq:xapprox}
	\begin{equation}\label{eq:xapprox_x}
	x(\phi)=x^{\rm approx}(\phi)+c\,e^{-L/2},\quad c=2\tanh(V/4)\exp\!\left[2\sinh^2(V/4)\right]+O(1/L)
	\end{equation}
	where
		\begin{equation}\label{eq:approx_xphi}
		x^{\rm approx}(\phi)=\begin{cases}
		\frac{1}{\sqrt{\tilde \alpha}}{\rm arcsinh}\left(\frac{\phi}{2 \tilde\varepsilon}\right), & 0\le \phi \leq \sqrt{\tilde\varepsilon},\\
		\frac{1}{\sqrt{ \tilde\alpha}}\left[\log\frac{4}{ \tilde\varepsilon}-2\,{\rm arccoth}\left(e^{\frac{\phi}2}\right)\right], & \phi > \sqrt{\tilde\varepsilon},
		\end{cases} 
		\end{equation}
		\begin{equation} \label{eq:sqrtaAnalytic}
		\sqrt{\tilde\alpha} =  \sqrt{1+\left[\frac{4\sinh^2(V/4)}{L}\right]^2}-\frac{4\sinh^2(V/4)}{L},\quad \tilde\varepsilon=4e^{-L\sqrt{\tilde\alpha}/2 }\, \tanh\left(\frac{V}4\right).
		\end{equation}
\end{subequations}
\end{lemma}
\begin{proof}
Direct application of Proposition \ref{prop:expansion_refined} to approximate~$I(\phi;\varepsilon)$ in relation \eqref{eq:intRepresentation2} yields
\[
x(\phi)=x^{\rm approx}(\phi)+O(\varepsilon/\sqrt{\alpha}),
\]
where
\begin{equation}\label{eq:approx_xphi_raw}
		x^{\rm approx}(\phi)=\begin{cases}
		\frac{1}{\sqrt{ \alpha}}{\rm arcsinh}\left(\frac{\phi}{2 \varepsilon}\right), & 0\le \phi \leq \sqrt{\varepsilon},\\
		\frac{1}{\sqrt{ \alpha}}\left[\log\frac{4}{ \varepsilon}-2\,{\rm arccoth}\left(e^{\frac{\phi}2}\right)\right], & \phi > \sqrt{\varepsilon}.
\end{cases} 
\end{equation}
Using approximations~\eqref{eq:approxForSqrta} and~\eqref{eq:firstRelation} for~$\alpha$ and~$\varepsilon$ in~\eqref{eq:approx_xphi_raw}, and neglecting~$O(\varepsilon)$ terms, yields~\eqref{eq:xapprox}. 
\end{proof}
\subsection{Numerical study} \label{sec:numericValid}
In this section, we present numerical simulations of the CCPB equation~\eqref{eq:CCPB} for the steady-state inverse solution $x(\phi)$ of the PNP system~\eqref{eq:PNP}.
 Throughout this section we use the numerical scheme described in~\cite{lee2010new}, unless otherwise stated.  

Figure~\ref{fig:phiProfiles}A presents a comparison between the inverse solution~$x(\phi)$ of~\eqref{eq:CCPB} ({\color{red} dashes}) and the corresponding approximate solution~$x^{\rm approx}(\phi)$, see~\eqref{eq:xapprox}, ({\color{blue}solid}) for~$V=5$ and for~$L=9$.  We observe that the approximation error is maximal around the point where~$\phi=\eta=\sqrt\varepsilon$, showing that as expected, the dominant source of error in the approximation of~$x(\phi)$ is the approximation error in~$I(\phi;\varepsilon)$, compare also with Figure~\ref{fig:crudeError}B.  
Figure~\ref{fig:phiProfiles}B presents the same data as Figure~\ref{fig:phiProfiles}A, but for~$L=15$.  In this case, the approximation error~$\|x(\phi)-x^{\rm approx}(\phi)\|_\infty\approx0.009$ (two curves are indistinguishable).   Finally, Figure~\ref{fig:phiProfiles}C presents the approximation error~$\|x(\phi)-x^{\rm approx}(\phi)\|_\infty$ for~$5< L< 40$ and shows that it agrees well with the predicted error~\eqref{eq:xapprox_x}. 
\begin{figure}[!h]
	\centering
	\includegraphics[width=\textwidth]{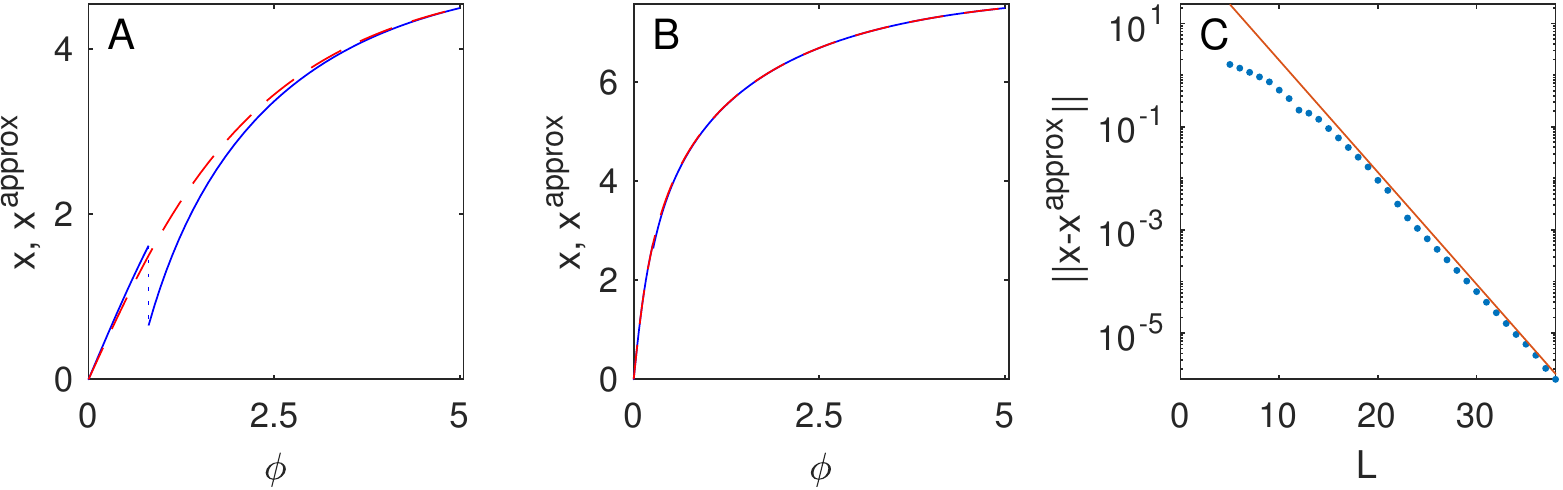}
	\caption{Comparison of the inverse solution~$x(\phi)$ of~\eqref{eq:CCPB} ({\color{red}dashes})  
and corresponding approximation~$x^{\rm approx}(\phi)$~\eqref{eq:xapprox} ({\color{blue}solid}) for $V=5$, and A: $L=9$, B: $L=15$.  Graph C presents the approximation error~$\|x(\phi)-x^{\rm approx}(\phi)\|_\infty$ as a function of~$L$ ({\color{blue}\tiny{$\bullet$}}) and predicted error~\eqref{eq:xapprox_x} ({\color{red} solid}).
}
	\label{fig:phiProfiles}
\end{figure}

\section{Finite domain effects}\label{sec:finiteDomainEffects}

Lemma~\ref{lem:explicitXOfPhi} provides an approximation of the inverse steady-state solution~$x(\phi)$ of the PNP equations~\eqref{eq:PNP} in a finite domain.\footnote{The analysis considers the solution behavior as a function of (rescaled) domain size~$L$ and (rescaled) applied voltage~$V$, while  the dependence upon average ion concentration~$\bar{c}$ is absorbed in the rescaling.  Accordingly, we restrict the interpretation of the non-dimensional results to the case when the average ion concentration is fixed and~$L$ or~$V$ vary.  Other case can be readily studied by considering dimensional variables or other scalings.}  Accordingly, it may be used to reveal when finite domain effects are significant and to quantify their nature.
\subsection{Distinct solution behaviors}\label{sec:distinct}
The approximation error~$E$ in~\eqref{eq:xapprox_x} is exponentially decreasing in~$L$,~$E=c e^{-L/2}$, but the coefficient~$c$ is exponential in~$V$, see~\eqref{eq:xapprox_x},
\begin{equation}\label{eq:error}
E\approx 2\tanh(V/4)e^{2\sinh^2(V/4)-L/2}.
\end{equation}
This error is proportional to~$\phi_x(0)$, see~\eqref{eq:sqrtaAnalytic} and~\eqref{eq:eps_def}.  
The point~$x=0$ is in the middle of the domain, farthest from the boundaries, i.e., at the bulk of the electrolyte solution.  Due to symmetry~$\phi(0)=0$.  If, additionally~$\phi_x(0)\ll1$, the solution can be regarded as electroneutral at the bulk.  When
\begin{equation}\label{eq:confinedDomain}
L\le 4\sinh^2(V/4)
\end{equation}
then~\eqref{eq:error} implies that~$E=O(1)$, and hence~$\phi_x(0)=O(1)$.  Thus, this case corresponds to solutions which do not reach electroneutrality at the bulk, see Figure~\ref{fig:illustrateProfiles}A.  At low voltages (which are at the focus of the classical theory of electrolytes), solutions reach bulk electroneutrality, unless they are in confined domains.   Accordingly, we refer to the parameter regime~\eqref{eq:confinedDomain} as the region corresponding to {\em confined domains}, but note that the domains may be relatively large when~$V$ is large.\begin{figure}[!h]
	\centering
	\includegraphics[width=0.9\textwidth]{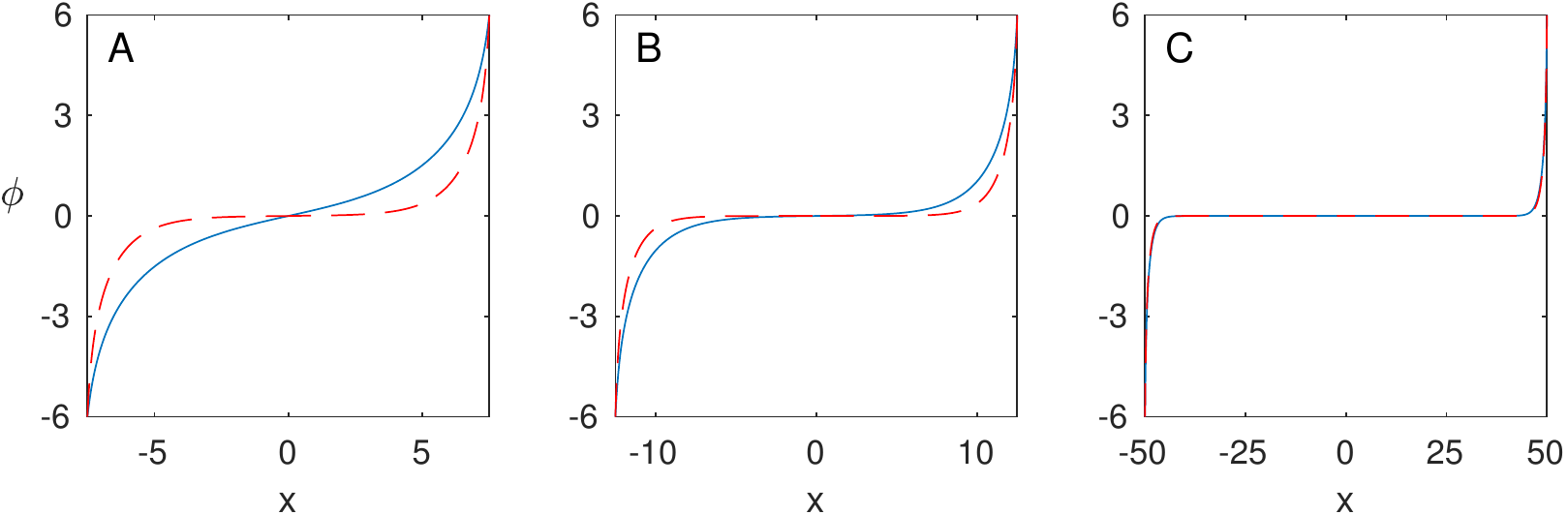}
	\caption{Steady-state solution~$\phi(x)$ of PNP~\eqref{eq:PNP} for~$V=6$ and numerous domain sizes~$L$ ({\color{blue} solid}), superimposed with the PNP solution~$\phi(x;V=6,L=\infty)$ in the domains~$[-L/2,0]$ and~$[0,L/2]$ ({\color{red} dashes}) for A:~$L=15$, B:~$L=25$ and~C:~$L=100$ (two curves are indistinguishable).}\label{fig:illustrateProfiles}
\end{figure}

For larger domain sizes the error~$E\ll1$, and also~$\phi_x(0)\ll1$.  In this case, the solution is nearly electroneutral at the bulk, see Figure~\ref{fig:illustrateProfiles}B and Figure~\ref{fig:illustrateProfiles}C.  In this region, the approximation~\eqref{eq:xapprox_x} is accurate and allows to quantify the effect of a finite domain.  Particularly, 
\begin{equation}
\begin{split}
x^{\rm approx}(\phi;V,L)&-x^{\rm approx}(\phi;V,L=\infty)=\\&\ln\left(4\frac{\cosh(V/4)}{\cosh(\phi/4)}\frac{\sinh(\phi/4)}{\sinh(V/4)}\right)\frac{4\sinh^2\!\left(\frac{V}4\right)}L+O\left(\frac1{L^2}\right).
\end{split}
\end{equation}
Therefore, finite domain effects are negligible when
\begin{equation}\label{eq:essentially_infinite}
\frac{4\sinh^2\!\left(\frac{V}4\right)}L\ll1.
\end{equation}
In this case, the steady-state solution of PNP~\eqref{eq:PNP} is well approximated by the solution of the Poisson-Boltzmann equation~\eqref{eq:PB}, see Figure~\ref{fig:illustrateProfiles}C.

Overall, we identify three parameter regimes, as depicted in Figure~\ref{fig:diagram},
\begin{itemize}
\item Region A: Corresponding to confined domains in which the solutions do not reach electro-neutrality at the bulk.  This region resides in the regime~\eqref{eq:confinedDomain}. 
\item Region B: Corresponding to large domains in which the solutions reach electro-neutrality at the bulk, but finite-domain effects are significant near the boundaries. This region resides in the regime~$E\ll1$ and~$L=O(4\sinh^2(V/4))$.
\item Region C: Corresponding to large enough domains so that finite-domain effects are negligible.  
This region resides in the regime~\eqref{eq:essentially_infinite}.
\end{itemize}

\begin{figure}[!h]
	\centering
	\includegraphics[width=0.4\textwidth]{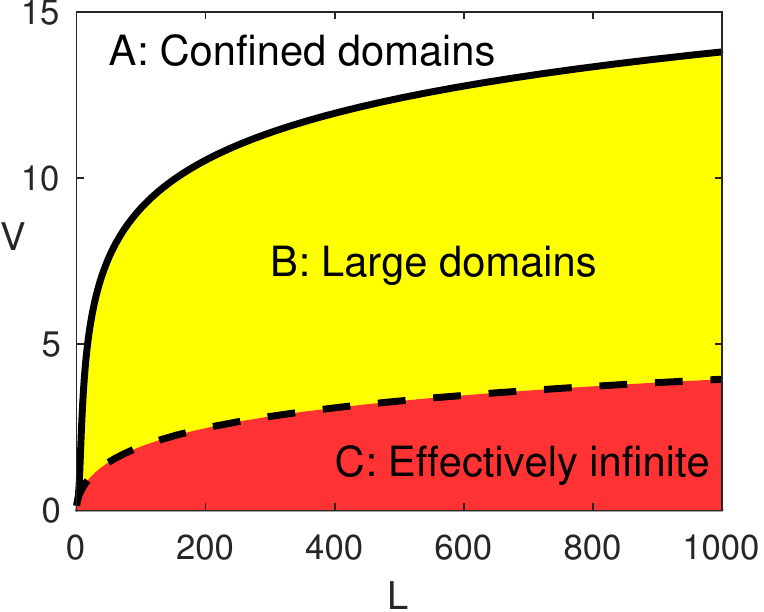}
	\caption{Graphs~$E={\rm tol}$ (solid) and~$\frac{4\sinh^2\!\left(\frac{V}4\right)}L={\rm tol}$ (dashed) for~${\rm tol}=0.05$.  These curves divide the parameter space into three distinct regions: Region $A$ corresponding to confined domains, see~\eqref{eq:confinedDomain}, region $C$ corresponding to large enough domains so that finite-domain effects are negligible, see~\eqref{eq:essentially_infinite}, and the intermediate region~$B$ in which finite-domain effects are significant near the boundaries. 
	}\label{fig:diagram}
\end{figure}

\subsection{Screening length} \label{sec:screening}
The screening length is an important measure of a charged boundary (or charged carrier) net electrostatic effect in an electrolyte solution.  When considering low charge on the boundary~$(V\ll1)$, it is commonly referred to as the Debye or Debye--H\"uckel length.  The screening length is defined as the distance from the electrode at which the electric potential decreases in magnitude by a factor of $\frac{1}{e}$.  Accordingly, the screening length~$\lambda_s$ predicted by the PNP model~\eqref{eq:PNP} is
\[
\lambda_s(L,V):=x(\phi=V;L,V)-x\left(\phi=\frac{V}{e};L,V\right),
\]
where~$x(\phi;L,V)$ is the inverse of the steady-state solution of the PNP system~\eqref{eq:PNP}.
Let us consider the ratio between the screening length~$\lambda_s(L,V)$ in a finite domain and~$\lambda_s(L=\infty,V)$ in a infinite domain
\[
\frac{\lambda_s(L,V)}{\lambda_s(L=\infty,V)}=\frac1{\sqrt{\alpha}}\frac{I(V;\varepsilon(L,V))-I(V/e;\varepsilon(L,V))}{I(V;\varepsilon(L=\infty,V))-I(V/e;\varepsilon(L=\infty,V))}=\frac1{\sqrt{\alpha}}+O(\varepsilon),
\]
where the last equality because~$I(V_1;\varepsilon)-I(V_2;\varepsilon)$ is, to leading order, independent of~$\varepsilon$, see~\eqref{eq:LemmaForI}.
Since $\alpha<1$, see~\eqref{eq:approxForSqrta}, we obtain that for any finite~$L>0$,
$\lambda_s(L)>\lambda_s(\infty)$. Therefore, finite domain effects increase the screening length.
Intuitively, screening of a surface charge involve the redistribution of counter-ions from the bulk to the vicinity of the surface.  Indeed,  the (normalized) ionic concentration in the bulk is~$\approx\alpha$, see~\eqref{eq:Boltzmann1}.
As a result, the counter-ion concentration in the bulk decreases, leading in turn to an increase in the entropic energy of the bulk.  Therefore, screening of a surface charge involves an energetic cost which becomes more dominant in a small domain, implying that the screening efficiency decreases with domain size.

Figure~\ref{fig:ScreeningComparisonV10} presents the ratio~$1/\sqrt{\tilde \alpha}\approx \lambda_s(V=10,L)/\lambda_s(V=10,L=\infty)$, see~\eqref{eq:sqrtaAnalytic}, as a function of the domain size, $L$.  As expected, as $L$ increases, the ratio decreases, tending to $1$ as $L\to \infty$.   Figure~\ref{fig:ScreeningComparisonV10}A-C present the electric potential profiles corresponding to point A-C in the top graph of Figure~\ref{fig:ScreeningComparisonV10}.  Even when~$L=300$, there are observable differences  between~$\phi(x;V,L)$ and~$\phi(x;V,L=\infty)$.  Thus,  finite domain effects, in this case, persist even at $L=300$.

\begin{figure}[ht!]
	\centering
	\includegraphics[width=0.85\textwidth]{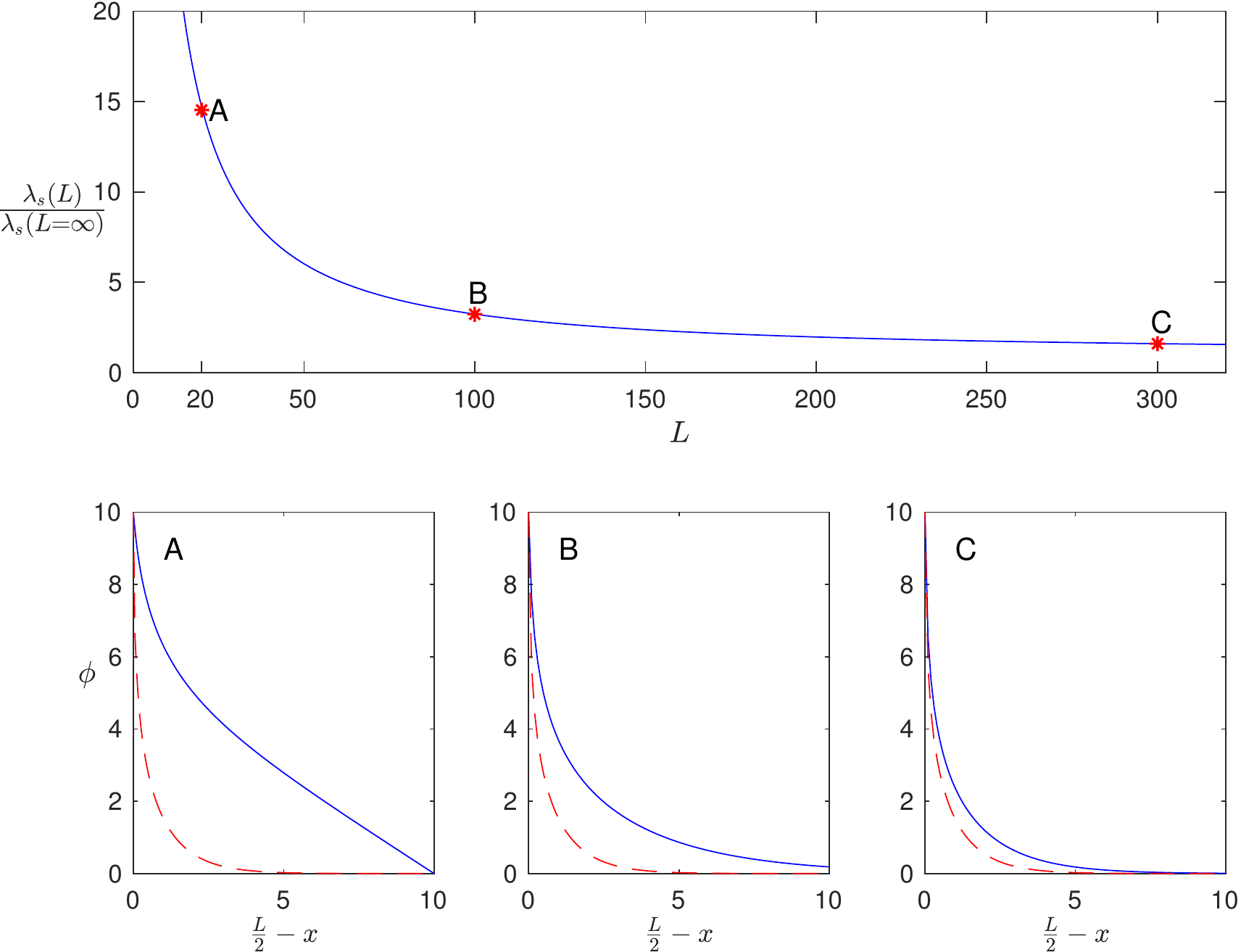}
	\caption{Top graph: The ratio~$\frac1{\sqrt{\tilde\alpha}}\approx\frac{\lambda_s(L,V=10)}{\lambda_s(L=\infty,V=10)}$ as a function of~$L$, where~$\lambda_s(L,V)$ is the screening length predicted by the PNP model~\eqref{eq:PNP} with given parameters~$L$ and~$V$.
Bottom graph present electric potential profiles~$\phi$ ({\color{blue}solid}) for $A:L=20$, $B:L=100$, and $C:L=300$.  {\color{red}Dashed} curve in the three bottom graphs is the PNP steady-state solution~$\phi(x;V=10,L=\infty)$.  All solutions are presented in the domain~$[L/2-10,L/2]$.}\label{fig:ScreeningComparisonV10}
\end{figure}

\section{Finite domain effects in generalized PNP models}	\label{sec:diffBC}
In this section, we apply the analysis of finite domain effects on steady-states of Poisson-Nernst-Planck equations to generalized PNP models.  In Section~\ref{sec:finiteDomainEffects}, different regimes of the parameter space were distinguished according to criteria arising from the detailed analysis of the PNP equations, see, e.g.,~\eqref{eq:essentially_infinite}.  We now suggest  alternative criteria that is not tailored to a specific PNP model, and therefore can be used to study finite domain effects in any generalized PNP equation.  

We identify three parameter regimes in which finite domain effects give rise to distinct solution behaviors, see Section~\ref{sec:finiteDomainEffects} and particularly Figure~\ref{fig:diagram}.  The first region corresponds to confined domains in which the solutions do not reach electro-neutrality at the bulk.  Consistent with the analysis of the PNP model, see Section~\ref{sec:distinct}, in a generalized PNP equation this region can be identified with~$\phi_x(0)=O(1)$.  The second region corresponds to large domains in which the solutions reach electro-neutrality at the bulk,~$|\phi_x(0)|\ll1$, but finite-domain effects are significant near the boundaries.  This region can be identified with the conditions
\[
|\phi_x(0)|\ll1,\quad |\phi(0;V,L)-\phi(0;V,L=\infty)|=O(1).
\]
Finally, finite-domain effects are considered negligible when~$|\phi(0;V,L)-\phi(0;V,L=\infty)|\ll1$.  These criteria are justified for any generalized PNP model that is not sensitive to changes in the bulk concentration.  

In what follows, we apply these criteria to study finite-domain effects in 
 the PNP-Stern model that accounts for a Stern layer of (normalized) width $\delta$, and is given by~\eqref{eq:CCPB} with
boundary conditions \cite{stern1924theorie,lee2010new,lee2014charge,lee2015boundary}
\begin{equation}\label{eq:SternBC}
\phi(-L/2)-\delta\phi_x(-L/2)=-V, \quad \phi(L/2)+\delta\phi_x(L/2)=V.
\end{equation}
Figures~\ref{fig:sternInfinite}A and~\ref{fig:sternInfinite}B presents the graphs~$\phi_x(0)={\rm tol}$ and~$|\phi(0;V,L)-\phi(0;V,L=\infty)|={\rm tol}$ for~$\delta=0$ and~$\delta=0.05$, respectively.  Similar to Section~\ref{sec:finiteDomainEffects}, the computation of these curves relies on the results Lemma \ref{lem:explicitXOfPhi} adapted to account for a Stern Layer.  Particularly, \Cref{lem:explicitXOfPhi} is also used to approximate $\phi_x(L/2)$.  We note, however, that these curves can be computed numerically, without relying on asymptotic analysis.

The PNP model assumes point charges, whereas the Stern layer is due to the finite size of the ions.  Therefore, a comparison between Figure~\ref{fig:sternInfinite}A and~\ref{fig:sternInfinite}B reveals how the finite size of the charges qualitatively impacts the model behavior in finite domains.  We observe that finite domain effects are observed at higher applied voltages~$V$, but that these effects are significant at domain sizes of similar orders of magnitude.

\begin{figure}[!h]
	\centering
	\includegraphics[width=0.9\textwidth]{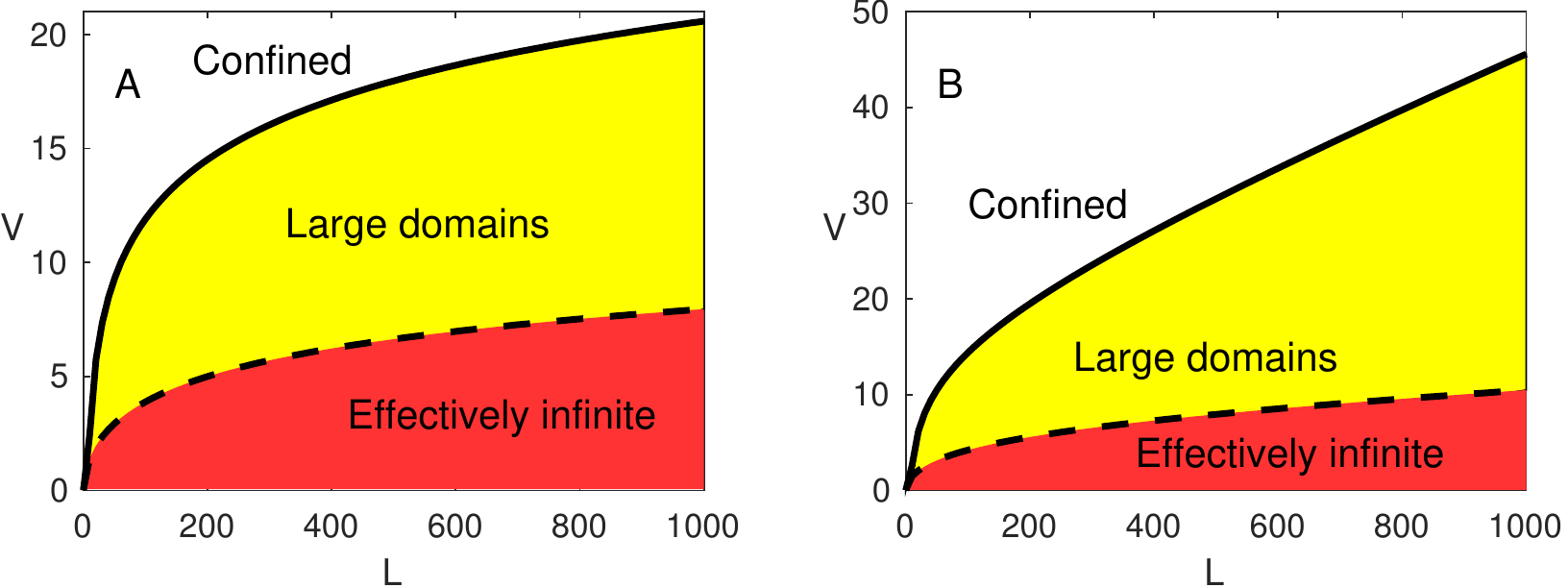}
	\caption{Graphs~$\phi_x(0)={\rm tol}$ (solid) and~$|\phi(0;V,L)-\phi(0;V,L=\infty)|={\rm tol}$ (dashed) for~${\rm tol}=0.05$ for steady solutions~$\phi(x)$ of PNP-Stern (\ref{eq:CCPB},\ref{eq:SternBC}) with A: $\delta=0$ and B:~$\delta=0.05$.
These curves divide the parameter space into three distinct regions corresponding to confined domains, large enough domains so that finite-domain effects are negligible, and an intermediate region in which finite-domain effects are significant near the boundaries.  
}\label{fig:sternInfinite}
\end{figure}

\section{Higher dimensions}\label{sec:highdim}
This study focuses on the study of finite domain effects in 
the one-dimensional case.  In particular, finite domain effects are significant when 
\[
|1-\alpha|\not\ll 1,
\]
see Section~\ref{sec:finiteDomainEffects}.

We now consider higher dimensions, and provide two concrete examples in which back-of-the-envelope computations based on the insights of this study are used to attain an estimate of the domain sizes at which finite-domain effects are significant.   Such estimates are useful, for example, to determine the size of a computational domain in numerical simulations or to better guide the modeling of systems with large but finite domains.

Following standard derivation and standard non-dimensionalization~\eqref{eq:scaledVars}, see, e.g.,~\cite{bazant2009towards}, the steady-state solution of the (non-dimensional) Poisson-Nernst-Planck equation in a domain~$\Omega\subset \mathbb{R}^d$ with no-flux boundary conditions and where the system is globally electroneutral
\begin{equation}\label{eq:globaly-electroneutral}
\frac1{|\Omega|}\int_\Omega p\,d{\bf x}=\frac1{|\Omega|}\int_\Omega n\,d{\bf x}=1,
\end{equation}
is given by 
\[
\nabla p+p\nabla \phi=0,\quad \nabla n-n\nabla \phi=0,\quad \nabla\cdot(\epsilon({\bf x})\nabla \phi)=\frac{n-p}2,
\]
where~$\epsilon({\bf x})$ is the relative dielectric constant.
Therefore,
\begin{subequations}\label{eq:Boltzmann_2D}
\begin{equation}
p=c e^{-\phi},\quad n=c e^{\phi}.
\end{equation}
Taking an average over both hands, and using~\eqref{eq:globaly-electroneutral} yields
\begin{equation}\label{eq:alpha_2D}
c=\alpha:=\left(\frac1{|\Omega|}\int_\Omega e^{\phi}dx\right)^{-1}=\left(\frac1{|\Omega|}\int_\Omega e^{-\phi}dx\right)^{-1}.
\end{equation}
\end{subequations}
When~$\alpha=c=1$, the charge distribution~\eqref{eq:Boltzmann_2D} identifies with the Boltzmann distribution attained at infinite domains.
Thus, similar to the one-dimensional case,~$|1-\alpha|$ is a measure for the magnitude of finite domain effects.  

\subsection*{Ion channels}
Ion channels are protein molecules that conduct ions (such as Na$^+$, K$^+$, Ca$^{2+}$, and Cl$^-$ that might be named bioions because of their universal importance in biology) through a narrow pore of fixed charge formed by the amino acids of the channel protein. Membranes are otherwise quite impermeable to natural substances, so channels are gatekeepers for cells and act as natural nano-valves.  Continuum mean-field theories of electrolytes, which are generalizations of Poisson-Nernst-Planck (PNP) models, have been widely used in studies of ion channels during the last two decades, for reviews see~\cite{eisenberg2012crowded,gillespie2015review,boda2014monte} and the references within.  These studies, however, focus almost solely on the channel while assuming the electrolyte bathes connected by the channel are large enough to neglect finite-domain effects.  

We now utilize the theory developed in this study to provide an estimate of the bath volume at which finite-domain effects may be significant.  The analysis applies to a general ion channel.  Nevertheless, in what follows, we adapt quantities that roughly correspond to RyR Calcium channels~\cite{fill2002ryanodine}.  We note that the analysis applied only to no-flux boundary conditions on all boundaries, while in many scenarios ion channel models take into account other boundary conditions.

Consider a narrow cylindrical channel connecting two electrolyte bathes.  The domain~$\Omega$ is therefore the union of the channel region~$\Omega_{\rm channel}$ and the left and right bath regions~$\Omega_{\rm bath}^L$ and~$\Omega_{\rm bath}^R$, respectively,
\[
\Omega=\Omega^L_{\rm bath} \cup\Omega_{\rm channel}\cup\Omega^R_{\rm bath}.
\]
The bath concentrations are taken to be~$\bar{c}_p=\bar{c}_n=0.1M$ (molar).  The permanent negative charge is assumed to be counter-balance by a positive charge concentration of~$18M$ inside the channel region, namely~$p_{\rm channel}/\bar{c}_p=180$.  
Under the simplifying assumption that the electric potential inside the channel~$\phi_{\rm channel}$ is uniform (aka resides in the Donnan equilibrium), then by~\eqref{eq:Boltzmann_2D},
\begin{equation}\label{eq:p_channel}
\frac{p_{\rm channel}}{\bar{c}_p}=180=\alpha\,e^{-\phi_{\rm channel}},\end{equation}
where the reference potential is taken to be~$\phi=\phi_{\rm bath}=0$ at the bath regions.

Finite-domain effects are negligible when~$\alpha\approx1$.  In this case,
\begin{equation*}
\begin{split}
\alpha^{-1}=\frac1{|\Omega|}\int_{\Omega}e^{-\phi} d{\bf x}&=\frac2{|\Omega|}\int_{\Omega_{\rm bath}}e^{-\phi_{\rm bath}} d{\bf x}+\frac1{|\Omega|}\int_{\Omega_{\rm channel}}e^{-\phi_{\rm channel}} d{\bf x}\\&=2\frac{|\Omega_{\rm bath}|}{|\Omega|}+\frac{180}{\alpha}\frac{|\Omega_{\rm channel}|}{|\Omega|},
\end{split}
\end{equation*}
where the last equality is due to the choice of the reference pontential~$\phi=\phi_{\rm bath}=0$, and due to~\eqref{eq:p_channel}.
Hence,
\[
\alpha^{-1}=\frac1{|\Omega|}\int_{\Omega}e^{-\phi} d{\bf x}=1+\frac{179}2\delta+O(\delta^2) \,\qquad \delta=\frac{|\Omega_{\rm channel}|}{|\Omega_{\rm bath}|}.
\]
Therefore, for example, to maintain an `error' below~$1\%$, namely to assume that the charge concentration identifies with the Boltzmann distribution up to a~$1\%$ error, one needs to choose~$\delta$ such that~$\frac{179}2\delta<{\rm error}=0.01$.  Therefore,
the bath dimension should be roughly $1/\delta\approx9\cdot 10^3$ times larger then the channel dimension.  

\subsection*{Porous electrodes}
Let us assume two porous electrodes separated by a bulk region.
The domain~$\Omega$, which is accessible to the electrolyte, is therefore the union of the bulk region~$\Omega_{\rm bulk}$, and the
two electrode regions~$\Omega_{\rm electrode}^-$ and~$\Omega_{\rm electrode}^+$,
\[
\Omega=\Omega^+_{\rm electrode} \cup\Omega_{\rm bulk}\cup\Omega^-_{\rm electrode}.
\]
Under the simplifying assumption that the electric potential inside the electrode pores~$\phi^\pm_{\rm electrode}$ is uniform (aka resides in the Donnan equilibrium), and by choosing the reference potential to be~$\phi=\phi_{\rm bulk}=0$, one attains
\begin{equation*}
\begin{split}
\alpha^{-1}&:=\frac1{|\Omega|}\int_{\Omega}e^{-\phi} d{\bf x}\\&=\frac1{|\Omega|}\int_{\Omega_{\rm bulk}}e^{-\phi_{\rm bulk}} d{\bf x}+\frac1{|\Omega|}\int_{\Omega^-_{\rm electrode}}e^{-\phi^-_{\rm electrode}} d{\bf x}+\frac1{|\Omega|}\int_{\Omega^+_{\rm electrode}}e^{-\phi^+_{\rm electrode}} d{\bf x}\\&=
\frac{|\Omega_{\rm bulk}|}{|\Omega|}+\frac{|\Omega^-_{\rm electrode}|}{|\Omega|}e^{-\phi^-_{\rm electrode}}+\frac{|\Omega^+_{\rm electrode}|}{|\Omega|}e^{-\phi^+_{\rm electrode}}.
\end{split}
\end{equation*}
In what follows, we further assume the electrodes width and porosity is the same~$|\Omega^-_{\rm electrode}|=|\Omega^+_{\rm electrode}|$, and consider the case~$\phi^+_{\rm electrode}=-\phi^-_{\rm electrode}$.  Thus,
\begin{equation*}
\begin{split}
\alpha^{-1}&=
\frac{|\Omega_{\rm bulk}|}{|\Omega|}+2\frac{|\Omega_{\rm electrode}|}{|\Omega|}\cosh(\phi_{\rm electrode}).
\end{split}
\end{equation*}

Finite domain effects are negligible when~$\alpha\approx1$. 
To ensure that~$\alpha>1-\delta$,
\[
\frac{|\Omega_{\rm bulk}|}{|\Omega_{\rm electrode}|}>2\frac{(1-\delta)\cosh(\phi_{\rm electrode})-1}\delta.
\]
Thus, for example, if the Donnan potential within the pores is~$\phi^{\rm dimensional}_{\rm electrode}=0.25V$ and the porosity of the electrodes is 0.3, then to maintain an `error' below~$1\%$, $\delta = 0.01$, 
\[
\frac{|\Omega_{\rm bulk}|}{|\Omega_{\rm electrode}|}>0.3\cdot200\left[0.99-e^{-\frac{q}{k_BT}\phi^{\rm dimensional}_{\rm electrode}}\right]\approx 59.4.
\]
Namely, finite-domain effects are negligible when the separation length between the porous electrodes exceeds~$\approx$60 times the width of the electrodes.  In a capacitive deionization (CDI) device, the  separation length between the porous electrodes is comparable to the electrode width, hence finite domain effects are expected to be significant in such a case. 

\section{Conclusions}\label{sec:conc}

In this study, we present an approximation for the steady-state solution of the PNP model in the asymptotic limit of a large, but finite, domain. 
This approximation allows distinguishing between confined domains, large domains in which finite-domain effects are significant and yet larger domains in which finite-domain effects are negligible, and to quantify finite domain effects in large domains.  

Surprisingly, we found that even for relatively large domains, finite domain effects are significant.  For example, in a reasonable scenario (bulk concentration of~$\bar{c}=0.1M$ and applied voltage of~$0.5V$), finite domain effects are dominant up to a micron scale ($L=1000\lambda_D$).  This is in contrast to the common approach that assumes long-range forces persist only up to a few Debye lengths (nanometer scale) region beyond the charged surface, as occurs for low applied voltages~$V\ll1$~\cite{huckel1923theory}.
Particularly, the study suggests that finite domain effects may be significant in physical systems such as capacitive deionization cells \cite{biesheuvel2011theory,porada2013review},  submicron gap capacitors \cite{hourdakis2006submicron,wallash2003electrical,chen2006electrical}, and many applications of microfluidics \cite{whitesides2006origins}.

The reason finite domain effects are significant in a relatively large domain is that the PNP steady-state solution has non-local dependence under no-flux boundary conditions.  Namely, as ions concentrate near the boundary to screen charge, they are depleted from the interior of the domain.  This property is common to a wide family of generalized PNP models with no-flux boundary conditions.  Therefore, we expect  finite domain effects to be significant in  generalized PNP models.  Analysis of finite-domain effects in additional PNP-type models, as well as further consideration of applications, will be published elsewhere.

\appendix
	\section{Proof of Lemma~\ref{prop:expansion_refined}} \label{app:lem_expansion_refined}	
For all~$0<\eta<\phi$,~$I(\phi;\varepsilon)=I_1(\eta;\varepsilon)+I_2(\phi;\varepsilon)$, where
\begin{equation}\label{eq:app_I0I1}
I_1(\eta;\varepsilon)=\int_0 ^\eta \frac{dx}{\sqrt{ \sinh^2 \left( x/2 \right) + \varepsilon^2}} ,\quad I_2(\phi;\varepsilon)=\int_\eta^\phi \frac{dx}{\sqrt{ \sinh^2 \left( x/2 \right) + \varepsilon^2}}.
\end{equation}
To approximate~$I_1$, first note that a high-order expansion about~$x=0$ yields
\begin{equation}\label{eq:sinh_2nd_expansion}
\sinh^2 \left( \frac{x}2 \right)+\varepsilon^2=\left(\frac{x^2}4+\varepsilon^2\right)\left[1+\frac{x^4}{48(\frac{x^2}4+\varepsilon^2)}+\frac{O(x^6)}{48(\frac{x^2}4+\varepsilon^2)}\right],\quad 0<x\ll1.
\end{equation}
The expansion~\eqref{eq:sinh_2nd_expansion} is an asymptotic expansion and in particular, 
\begin{equation}\label{eq:explicitOx4}
\frac{x^4}{48(\frac{x^2}4+\varepsilon^2)}\ll1,
\end{equation}
for all~$x\ll1$.  Therefore, knowledge of the specific form of the correction terms allows relaxing constraints on~$\eta$ that were applied in the proof of Proposition~\ref{prop:expansion}.  The approximation of~$I_2$, detailed below, will require~$\varepsilon\ll\eta$, hence overall~$\varepsilon\ll\eta\ll1$.  In what follows, we consider~$\eta=\sqrt{\varepsilon}$.

Substituting~\eqref{eq:sinh_2nd_expansion} in~\eqref{eq:app_I0I1} yields
\begin{equation*}
\begin{split}
I_1(\phi;\varepsilon)&=\int_0 ^\phi \frac{dx}{\sqrt{\frac{x^2}4+\varepsilon^2}\sqrt{1+\frac{x^4+O(x^6)}{48(\frac{x^2}4+\varepsilon^2)}}}\\&=
\int_0^\phi \frac{1}{\sqrt{\frac{x^2}4+\varepsilon^2}}\left[1-\frac1{24}\frac{x^4+\tilde E_1(x)}{x^2+4\varepsilon^2}\right]dx
\\&=\left(2-\frac{\varepsilon^2}2\right)\,\mbox{arcsinh}\left(\frac{\phi}{2\varepsilon}\right) - \frac{\phi^3 + 12\phi\varepsilon^2}{24\sqrt{\phi^2 + 4 \varepsilon^2}}+O(\varepsilon^2)\\
\\&=2\,\mbox{arcsinh}\left(\frac{\phi}{2\varepsilon}\right) - \frac{\phi^2}{24}+O(\varepsilon^2),\quad 0<\phi<\eta=\sqrt{\varepsilon},
\end{split}
\end{equation*}
where the magnitude of the approximation error is attained by direct integration of~$I_1$ with an error term of the form~$\tilde E_1(x)=a_6 x^6+(a_8+b_8/(4x^2+\varepsilon^2))x^8+\cdots$.

To approximate~$I_2$ note that~$ \sinh x/2\gg\varepsilon$ when~$x\gg \varepsilon$.  In this case,
\[
\sqrt{\sinh^2 \frac{x}2+\varepsilon^2}=\sinh \frac{x}2 \sqrt{1+\frac{\varepsilon^2}{\sinh ^2\frac{x}2}}=\left(1+\frac{\varepsilon^2}{2\sinh ^2\frac{x}2}+O\left(\frac{\varepsilon^4}{x^4}\right)\right)\sinh \frac{x}2.
\]
Thus, 
\begin{equation*}
\begin{split}
I_2&=\int_\eta^\phi \frac{1}{\sinh \frac{x}2}\frac{1}{1+\frac{\varepsilon^2}{2\sinh ^2\frac{x}2}+O\left(\frac{\varepsilon^4}{x^4}\right)}dx=
\int_\eta^\phi \left[1-\frac{\varepsilon^2}{2\sinh ^2\frac{x}2}+O\left(\frac{\varepsilon^4}{x^4}\right)\right]\frac{dx}{\sinh \frac{x}2}
\\&=(4+\varepsilon^2)[{\rm arctanh}(e^{\eta /2})-{\rm arctanh}(e^{\phi/2})]+\frac{\varepsilon^2}2 \left(\frac{\cosh\frac{\phi}2}{\sinh^2\frac{\phi}2}-\frac{\cosh\frac{\eta}2}{\sinh^2\frac{\eta}2}\right)+O(\varepsilon^2),
\end{split}
\end{equation*}
and
\[
I_1+I_2=2\log\frac{4}{\varepsilon}-4{\rm arctanh}(e^{\phi/2})+\frac{\varepsilon^2}2 \frac{\cosh\frac{\phi}2}{\sinh^2\frac{\phi}2}+O(\varepsilon^2\log\varepsilon).
\]
Note that the error is~$O(\varepsilon)$ since for~$\phi=O(\eta)$
\[
\frac{\varepsilon^2}2 \frac{\cosh\frac{\phi}2}{\sinh^2\frac{\phi}2}=O(\varepsilon).
\]

\end{document}